\documentclass[11pt]{amsart}
\usepackage[margin=1.15in]{geometry}
\usepackage{amsmath,amssymb,amsthm,mathtools}
\usepackage{mathrsfs}
\usepackage{enumitem}
\usepackage{needspace}
\usepackage[expansion=false]{microtype}
\usepackage{xcolor}
\usepackage[colorlinks=true,linkcolor=blue!60!black,citecolor=blue!60!black,urlcolor=blue!60!black]{hyperref}
\usepackage{comment}

\setlist[enumerate]{leftmargin=*,itemsep=3pt,topsep=4pt}
\numberwithin{equation}{section}
\newtheorem{theorem}{Theorem}[section]
\newtheorem{lemma}[theorem]{Lemma}
\newtheorem{proposition}[theorem]{Proposition}
\newtheorem{corollary}[theorem]{Corollary}
\theoremstyle{definition}
\newtheorem{definition}[theorem]{Definition}
\newtheorem{example}[theorem]{Example}

\theoremstyle{remark}
\newtheorem{remark}[theorem]{Remark}
\theoremstyle{definition}
\newtheorem*{AC}{Acknowledgement}
\newtheorem*{AS}{AI statements}
\newcommand{\Dmod}{\mathcal D}
\newcommand{\Amod}{\mathcal A}
\newcommand{\Cring}{\mathbb C}
\newcommand{\Mmod}{\mathcal M}
\newcommand{\svect}{\mathbf s}
\newcommand{\tvect}{\mathbf t}
\newcommand{\one}{\mathbf 1}
\newcommand{\Chrel}{\operatorname{Ch}^{\mathrm{rel}}}
\newcommand{\CCrel}{\operatorname{CC}^{\mathrm{rel}}}
\newcommand{\Ann}{\operatorname{Ann}}
\newcommand{\Supp}{\operatorname{Supp}}
\newcommand{\Spec}{\operatorname{Spec}}
\newcommand{\CC}{\operatorname{CC}}
\newcommand{\gr}{\operatorname{gr}}
\newcommand{\Div}{\operatorname{Div}}
\newcommand{\Exp}{\operatorname{Exp}}

\title[Bounds for Bernstein--Sato zero loci]{Bounds for codimension-one components of zero loci of Bernstein--Sato ideals}

\author{Wenzong Guo}
\address{Wenzong Guo, School of Mathematical Sciences, Zhejiang University, Hangzhou, China}
\email{wenzongguo@zju.edu.cn}
\author{Fanghan Xiang}
\address{Fanghan Xiang, School of Mathematical Sciences, Zhejiang University, Hangzhou, China}
\email{fanghanxiang@zju.edu.cn}

\begin{document}
\begin{abstract}
Let $X$ be a smooth complex affine variety of dimension $n$, and let $F=(f_1,\ldots,f_r)$ be a tuple of nonzero regular functions on $X$ such that their product $f:=\prod_{i=1}^r f_i$ is not invertible. We study the zero loci of the Bernstein--Sato ideals $B_F^{\mathbf a}$ for nonnegative integral shifts $\mathbf a$. For a fixed log resolution, every codimension-one irreducible component of $Z(B_F^{\mathbf{a}})$ is a hyperplane of the form $L_E(\svect)+k_E+c=0$ with $c$ a positive integer. We give a new proof of this theorem using localized maximal and minimal extensions of relative $\Dmod$-modules. We also prove that
$c\leq L_E(\mathbf a)+(n-1-\delta_f)L_E(\one)-k_E$,
where $\delta_f=\min\{n-1,\alpha_f\}$ and $\alpha_f$ is the minimal exponent of $f$. The problem of obtaining such an upper bound for arbitrary tuples (in particular, for $r>1$) was raised in \cite[\S1]{BudurVanDerVeerVanWerde2024}, and the displayed inequality resolves it. To obtain the stated upper bound, we compare the diagonal slice of \(Z(B_F^{\mathbf 1})\) with the root set of \(b_f\). The finite covering of $Z(B_F^{\mathbf a})$ by translates of its codimension-one part along vectors $-k\mathbf a$, combined with diagonal specialization and the log-resolution description, shows that these two sets have the same least and greatest points. In particular, applying Saito’s root estimate to their common least point then yields the desired upper bound. We further establish a divisor-valued formulation of the local index comparison, recovering both the detection of monodromy support via monodromy zeta functions and the multivariable A’Campo formula.
\end{abstract}
\maketitle

\setcounter{tocdepth}{1}
\tableofcontents

\section{Introduction}

\subsection{Bernstein--Sato ideals and log resolutions}
Let $X$ be a smooth complex affine algebraic variety of dimension $n\geq1$, and let $F=(f_1,\ldots,f_r)$ be a tuple of nonzero regular functions on $X$. We write
\[
R=\Cring[\svect]=\Cring[s_1,\ldots,s_r],\qquad
\Dmod_X[\svect]=\Dmod_X\otimes_{\Cring}R,\qquad
F^{\svect}=\prod_{i=1}^r f_i^{s_i},
\]
where $\Dmod_X$ is the sheaf of algebraic differential operators on $X$. Put $f=\prod_i f_i$ and $D=(f=0)$, and assume that $f$ is not invertible. For $\mathbf a\in\mathbb Z_{\geq0}^r\setminus\{0\}$, the \emph{Bernstein--Sato ideal} of $F$ with shift $\mathbf a$ is
\[
B_F^{\mathbf a}
=\Ann_R\left(\frac{\Dmod_X[\svect]F^{\svect}}
                         {\Dmod_X[\svect]F^{\svect+\mathbf a}}\right).
\]
Equivalently, it consists of the polynomials $b(\svect)\in R$ satisfying
\[
b(\svect)F^{\svect}\in\Dmod_X[\svect]F^{\svect+\mathbf a}.
\]
When $r=1$ and $a=1$, we recover the one-variable Bernstein–Sato ideal, which is principal and generated by the usual Bernstein–Sato polynomial $b_f(s)$.

\begin{theorem}[{\cite[Theorem 1.1.1]{BvdVWZ21b}}]
\label{thmbvdwz}
Let $F=(f_1,\ldots,f_r):X\to\mathbb C^r$ be a morphism
from a smooth complex affine irreducible algebraic variety $X$,
or the germ at $x\in X$ of a holomorphic map on a complex manifold.
Let $\mathbf a=(a_1,\ldots,a_r)\in\mathbb Z_{\geq0}^r$ be such that
$\prod_{j=1}^r f_j^{a_j}$ is not invertible as a holomorphic
function on $X$. Then the following hold.
\begin{enumerate}
    \item Every irreducible component of $Z(B_F^{\mathbf a})$
    of codimension one is a hyperplane of the form
    \[
        l_1s_1+\cdots+l_rs_r+b=0,
    \]
    where $l_j\in\mathbb Q_{\geq0}$ and $b\in\mathbb Q_{>0}$.
    For each such hyperplane, there exists an index $j$
    with $a_j\neq0$ such that $l_j>0$.

    \item Every irreducible component of $Z(B_F^{\mathbf a})$
    of codimension greater than one can be translated
    by an element of $\mathbb Z^r$ into a codimension-one
    irreducible component of $Z(B_F^{\mathbf a})$.
\end{enumerate}
\end{theorem}

For $r=1$, assertion~(1) is
equivalent to Kashiwara's rationality theorem that the roots of the
Bernstein--Sato polynomial $b_f$ are negative rational numbers
\cite{Kas76}. Assertion~(1), without the additional requirement
that $l_j>0$ for some $j$ with $a_j\neq0$, is due to Sabbah
\cite{Sab87I} and Gyoja \cite{Gyo93}. In the case
$\mathbf a=\mathbf 1$, assertion~(2) was established by Maisonobe
\cite{Mai23}; a different proof was subsequently given by
van der Veer \cite{vanderVeer2021}. Budur, van der Veer, and Van Werde
\cite{BudurVanDerVeerVanWerde2024} refined assertion~(1)
by describing the codimension-one components in terms of
the numerical data of a log resolution.

Fix a log resolution $\mu:Y\to X$ of $(X,D)$ which is an isomorphism over $X\setminus D$. Let $\mathcal E$ be the set of irreducible components of $\mu^{-1}D$. For $E\in\mathcal E$, set
\begin{equation}\label{eq:intro-resolution-data}
N_{E,i}=\operatorname{ord}_E(f_i\circ\mu),\qquad
L_E(\svect)=\sum_{i=1}^r N_{E,i}s_i,\qquad
k_E=\operatorname{ord}_E(K_{Y/X}).
\end{equation}
Thus $k_E$ is the order of the Jacobian determinant of $\mu$ along $E$, and $L_E(\one)>0$, where $\one=(1,\ldots,1)$. One has:

\begin{theorem}[{\cite[Theorem 1.2]{BudurVanDerVeerVanWerde2024}}]\label{Theorem:1.1}
Every irreducible component of $Z(B_F^{\mathbf a})$ of codimension one is a hyperplane
\[
L_E(\svect)+k_E+c=0
\]
for some $E\in\mathcal E$ and some $c\in\mathbb Z_{>0}$.
\end{theorem}
Theorem~1.2 gives the lower bound $c\geq 1$.
Our first aim is to give a new proof of Theorem~1.2 using localized relative $\mathcal D_X$-modules and our proof adapts Wu's argument for Lichtin's theorem \cite[Theorem 4.1]{WuStrongMonodromy2026} to the multivariable setting.

\subsection{The upper bound}
In the one-variable case, an upper bound of $c$ follows from Saito's estimate for the roots of the microlocal $b$-function. The problem of obtaining an upper bound for arbitrary tuples was raised in \cite[\S1]{BudurVanDerVeerVanWerde2024}. In this paper, we give such a bound in terms of the minimal exponent of the product $f$.

Let
$\widetilde b_f(s)=\frac{b_f(s)}{s+1}$
be the microlocal $b$-function. we write $$\alpha_f=\min Z\bigl(\widetilde b_f(-s)\bigr),\quad \delta_f=\min\{n-1,\alpha_f\}.$$ Saito's theorem \cite[Theorem 0.4]{Saito1994Microlocal} gives
$Z\bigl(\widetilde b_f(-s)\bigr)\subseteq[\alpha_f,n-\alpha_f]$.
Since $Z(b_f)=Z(\widetilde b_f)\cup\{-1\}$, we obtain
\begin{equation}\label{eq:intro-root-bound}
\rho\geq-n+\delta_f\qquad\text{for every }\rho\in Z(b_f).
\end{equation}

\begin{theorem}\label{Theorem:1.2}
Let $H$ be an irreducible component of $Z(B_F^{\mathbf a})$ of codimension one. If
\[
H=Z\bigl(L_E(\svect)+k_E+c\bigr),\qquad c\in\mathbb Z_{>0},
\]
then
\begin{equation}\label{eq:main-upper-bound}
c\leq L_E(\mathbf a)+(n-1-\delta_f)L_E(\one)-k_E.
\end{equation}
\end{theorem}

Together with Theorem~\ref{Theorem:1.1}, this gives a finite interval of possible integers $c$ for each component $E$ of the chosen resolution. When $r=1$, the estimate becomes $c\leq(n+a-1-\delta_f)N_E-k_E$, which is the bound obtained from the one-variable root estimate. When $n=1$, it becomes $c\leq L_E(\mathbf a)$; this bound is attained by tuples of monomials on a smooth curve. These cases are discussed in Section~\ref{sec:sharpness}.

\subsection{Positive translations and the diagonal interval}
Write $W_{\mathbf a}=Z(B_F^{\mathbf a})$, and let $W_{\mathbf a}^{(1)}$ be the union of its codimension-one irreducible components. The following refinement of Theorem~\ref{thmbvdwz} (2) is well known to experts, but we are not aware of an explicit reference in the literature. We include
a proof for completeness.

\begin{theorem}\label{thm:intro-positive-translation}
Let $\mathbf a\in\mathbb Z_{\geq0}^r$ with $F^{\mathbf a}$ not invertible. There is an integer $q\geq0$ such that
\begin{equation}\label{eq:intro-positive-cover}
W_{\mathbf a}\subseteq\bigcup_{k=0}^{q}
\bigl(W_{\mathbf a}^{(1)}-k\mathbf a\bigr).
\end{equation}
In particular, for every irreducible component $C$ of $W_{\mathbf a}$ of codimension at least two, there are $k\geq1$ and a codimension-one component $H$ of $W_{\mathbf a}$ such that $C+k\mathbf a\subseteq H$.
\end{theorem}

For the simultaneous shift $\mathbf{a}=\mathbf{1}$, consider the diagonal slice
\[
S=\{t\in\Cring:t\one\in Z(B_F^{\one})\},
\]
Theorem~\ref{thm:intro-positive-translation} controls the upper endpoint of this slice, while diagonal specialization controls its lower endpoint. We write $\rho_-:=\min Z(b_f)$ and $\rho_+:=\max Z(b_f)$ for the least
and greatest roots of $b_f$ respectively.

\begin{theorem}\label{thm:intro-diagonal-interval}
The set $S$ is a finite subset of $\mathbb Q_{<0}$, and
\begin{equation}\label{eq:intro-diagonal-interval}
Z(b_f)\subseteq S\subseteq[\rho_-,\rho_+].
\end{equation}
Every element of $t\in S$ is of the form $t=\rho+m$ where $\rho \in Z(b_f), m\in \mathbf{N}$. In particular, $\min S=\rho_-$ and $\max S=\rho_+$.
\end{theorem}

The largest root satisfies $\rho_+=-\operatorname{lct}_X(f)$ ( see \cite[Theorem 2]{BudurMustataSaito2006}). Thus the smallest closed real interval containing the diagonal slice is exactly the interval between the smallest root of $b_f$ and the negative log canonical threshold of $f$.

\subsection{Proof strategy}
The upper bound is obtained by comparing a finite diagonal slice of $Z(B_F^{\one})$ with the roots of $b_f$. Put
\[
\Delta=\{t\one:t\in\Cring\},\qquad
\nabla(\svect)=\svect+\one.
\]
The diagonal specialization theorem of van der Veer \cite[Theorem F]{vanderVeer2021} states, for tuples on affine space, that
\[
\bigl(Z(B_F^{\one})\cap\Delta\bigr)
\setminus\bigcup_{m>0}\nabla^m\bigl(Z(B_F^{\one})\bigr)
\subseteq Z(b_f),
\]
where $Z(b_f)$ is identified with a subset of $\Delta$. In Section~\ref{sec:specialization}, we extend this statement to smooth affine varieties using relative Kashiwara equivalence and van der Veer's fiber nonvanishing theorem\cite[Theorem E]{vanderVeer2021}. Finiteness of $S$ then shows that every $t\in S$ is a nonnegative integral translate of a root of $b_f$. This proves the lower inequality in Theorem~\ref{thm:intro-diagonal-interval}. Combining it with~\eqref{eq:intro-root-bound} and a filtration by unit shifts gives~\eqref{eq:main-upper-bound} for every $\mathbf a$.

To prove Theorem~\ref{thm:intro-positive-translation}, we use the maximal tame pure extension of $\Dmod_X[\svect]F^{\svect}$ from \cite[\S4.4]{BvdVWZ21b}. Its quotient by the $\mathbf a$-shift has parameter support equal to $W_{\mathbf a}^{(1)}$. Stabilization of the kernels of the induced semilinear shift on the extension quotient gives the finite covering~\eqref{eq:intro-positive-cover}. Applying this covering with $\mathbf a=\one$ moves each diagonal point in the positive direction to a log-resolution hyperplane. The bound $c\geq1$ then gives $t\leq-\operatorname{lct}_X(f)$, proving the upper inequality in Theorem~\ref{thm:intro-diagonal-interval}. Finally, specializing functional equations gives $Z(b_f)\subseteq S$, so both endpoints are attained.

We also give a proof of Theorem~\ref{Theorem:1.1} using localized relative $\Dmod$-modules. Let $j:U=X\setminus D\hookrightarrow X$. For a prime ideal $\mathfrak p\subset R$, the maximal and minimal extensions of $\mathcal O_U[\svect]_{\mathfrak p}F^{\svect}$ are coherent, and their quotient is
\[
\Psi_{F,\mathfrak p}(\mathcal O_X)
=\frac{j_*(\mathcal O_U[\svect]_{\mathfrak p}F^{\svect})}
        {j_!(\mathcal O_U[\svect]_{\mathfrak p}F^{\svect})}
=\left(\frac{\Dmod_X[\svect]F^{\svect-k\one}}
              {\Dmod_X[\svect]F^{\svect+k\one}}\right)_{\mathfrak p}
\quad(k\gg0).
\]
These extensions were studied in \cite{Wu2022,Wu26} and the construction is a relative analogue of the algebraic approach to nearby cycles in \cite[\S4.2]{BeilinsonBernstein1993}. On a log resolution, the cyclic lattice generated by $F_Y^{\svect}\mu^*(dx)$ agrees with the minimal extension away from the hyperplanes in Theorem~\ref{Theorem:1.1}. Proper direct image along $\mu$ then places the original generator $F^{\mathbf s}dx$ in the minimal extension on \(X\), forcing the relevant localized shift quotient to vanish.

\subsection{Monodromy zeta functions}
The same localized extensions also give a description of monodromy zeta functions in terms of characteristic cycles. Let
\[
A=\Cring[t_1^{\pm1},\ldots,t_r^{\pm1}],\qquad
\mathbb T=\Spec A,
\]
and let $\psi_F(\Cring_X)$ be the Sabbah specialization complex. We denote its support in $\mathbb T$ by $S(F)$. In Section~\ref{monodromy}, we use 
$\psi_F(\Cring_X)$
and its divisor-valued stalk Euler characteristic $\chi_{\mathrm{st}}^A(\psi_F(\Cring_X))$.

For $x\in X$, let $\zeta_{F,x}^{\mathrm{mon}}$ denote the rational function, defined up to a unit of $A$, whose divisor is $\chi_{\mathrm{st}}^A(\psi_F(\Cring_X))(x)$. Write $PZ(\zeta_{F,x}^{\mathrm{mon}})$ for the union of its zero and pole divisors. The local index comparison in \cite[\S1.3]{Wu26},
when expressed through the characteristic-cycle isomorphism,
yields a new proof of the following results.

\begin{theorem}\label{thm:intro-monodromy}
With the preceding normalization,
\[
S(F)=\bigcup_{x\in X}PZ(\zeta_{F,x}^{\mathrm{mon}}).
\]
Moreover, for every $x\in X$,
\begin{equation}\label{eq:intro-acampo}
\zeta_{F,x}^{\mathrm{mon}}(\tvect)
\doteq\prod_{E\in\mathcal E}
\bigl(\tvect^{N_E}-1\bigr)^{-\chi_c(E^\circ\cap\mu^{-1}(x))},
\end{equation}
where $\doteq$ denotes equality up to a unit of $A$, $N_E=(N_{E,1},\ldots,N_{E,r})$, $E^\circ=E\setminus\bigcup_{E'\neq E}E'$.
\end{theorem}

For $r=1$, the first assertion is due to Denef
\cite[Lemma 4.6]{Den93}, while the second is the classical
A'Campo formula \cite{AC75}. In the multivariable case,
the first assertion was proved by Budur, Liu, Saumell, and Wang
\cite[Theorem 1.3]{BLSW17}, and the second is the multivariable
A'Campo formula; see \cite{S90}.

\subsection{Organization}
The paper is organized as follows.
Section~\ref{RELD} collects the results on relative
$\Dmod$-modules used in the proofs.
Section~\ref{lower} proves Theorem~\ref{Theorem:1.1}.
Section~\ref{sec:specialization} establishes the diagonal
specialization result for smooth affine varieties and proves
Theorems~\ref{Theorem:1.2},
\ref{thm:intro-positive-translation}, and
\ref{thm:intro-diagonal-interval}.
Section~\ref{monodromy} develops the divisor-valued formulation
and proves Theorem~\ref{thm:intro-monodromy}.

\begin{AC}
We would like to thank our advisor, Lei Wu, for introducing us to this topic and for his valuable suggestions and guidance. We are particularly indebted to him for teaching us how to prove Theorems~\ref{Theorem:1.1} and ~\ref{thm:intro-monodromy}.
\end{AC}
\begin{AS}
One of the results presented in this paper,
Theorem~\ref{Theorem:1.2}, was obtained with assistance
from an automated reasoning system.
Danus~\cite{LiuEtAl2026Danus}, running GPT-5.6 Sol,
generated proofs of Proposition~\ref{Proposition:1} and
Lemma~\ref{lem:unit-shift-bound} and proposed how these
two results could be combined to establish the theorem.
We subsequently verified all mathematical arguments
and revised the proofs and their presentation.
We take full responsibility for the mathematical content
and the final version of this paper.
\end{AS}

\section{Relative \texorpdfstring{$\Dmod$}{D}-modules}\label{RELD}

We recall the relative $\Dmod$-module formalism needed below, following \cite{Wu2022}. Unless otherwise stated, $X$ is smooth of dimension $n$, and $R$ is a regular commutative noetherian $\Cring$-algebra which is an integral domain of finite Krull dimension. The coefficient rings used below are polynomial rings over $\Cring$ and their localizations. Put $\Amod_R=\Dmod_X\otimes_{\Cring}R$. All modules are left modules unless a right action is indicated.

\subsection{Characteristic varieties and parameter support}
A coherent $\Amod_R$-module $\Mmod$ admits locally a good filtration for the relative order filtration
\[
F_k^{\mathrm{rel}}\Amod_R=F_k\Dmod_X\otimes_{\Cring}R.
\]
The elements of $R$ have degree zero. The associated graded module is coherent over $\gr\Dmod_X\otimes_{\Cring}R$ and determines a coherent sheaf on $T^*X\times\Spec R$. Its support is the \emph{relative characteristic variety}
\[
\Chrel(\Mmod)=\Supp(\gr^{\mathrm{rel}}\Mmod).
\]
The corresponding cycle, with multiplicities given by lengths at generic points, is denoted by $\CCrel(\Mmod)$. Both are independent of the chosen good filtration.

\begin{definition}
A coherent $\Amod_R$-module $\Mmod$ is \emph{relative holonomic} if every irreducible component of $\Chrel(\Mmod)$ is of the form $\Lambda\times S$, where $\Lambda\subseteq T^*X$ is an irreducible conic Lagrangian subvariety and $S\subseteq\Spec R$ is an irreducible closed subvariety.
\end{definition}

Relative holonomic modules form an abelian category \cite[\S3.2]{BudurVanDerVeerWuZhou2021}. For a coherent $\Amod_R$-module $\Mmod$, write
\[
B_{\Mmod}=\Ann_R(\Mmod),\qquad
\Supp_R(\Mmod)=\{\mathfrak p\in\Spec R:\Mmod_{\mathfrak p}\neq0\}.
\]
Let $p_2:T^*X\times\Spec R\to\Spec R$ be the projection.

\begin{lemma}[{\cite[Lemma 3.4.1]{BudurVanDerVeerWuZhou2021}}]\label{lem:relative-support}
For a relative holonomic $\Amod_R$-module $\Mmod$,
\[
Z(B_{\Mmod})=p_2\bigl(\Chrel(\Mmod)\bigr)=\Supp_R(\Mmod).
\]
\end{lemma}

The last equality requires care because $\Mmod$ need not be finitely generated over $R$. We use Lemma~\ref{lem:relative-support} whenever we pass between annihilators and parameter support.

Localization in $R$ is exact and commutes with taking a good filtration and its associated graded module. In particular,
\begin{equation}\label{eq:relative-localization}
\CCrel(\Mmod_{\mathfrak p})=\CCrel(\Mmod)_{\mathfrak p}.
\end{equation}
For a short exact sequence $0\to\Mmod'\to\Mmod\to\Mmod''\to0$,
\[
\Chrel(\Mmod)=\Chrel(\Mmod')\cup\Chrel(\Mmod'').
\]
When the nonzero terms have characteristic varieties of the same pure dimension, their characteristic cycles are additive. This is the situation in the localized normal-crossings computation in Section~\ref{sec:mono-nc}.

\subsection{Homological properties}
For a nonzero coherent $\Amod_R$-module $\Mmod$, its \emph{grade} is
\[
j(\Mmod)=\min\{k\geq0:
\mathcal Ext^k_{\Amod_R}(\Mmod,\Amod_R)\neq0\}.
\]
The module is \emph{$j$-pure} if every nonzero submodule has grade $j$, and it is \emph{$j$-Cohen--Macaulay} if
\[
\mathcal Ext^k_{\Amod_R}(\Mmod,\Amod_R)=0\qquad(k\neq j).
\]
We use the dimension formula
\begin{equation}\label{eq:grade-dimension}
j(\Mmod)+\dim\Chrel(\Mmod)=2n+\dim R
\end{equation}
and the fact that a $j$-Cohen--Macaulay module is $j$-pure; see \cite[Lemma 3.2.2 and Remark 3.3.2]{BudurVanDerVeerWuZhou2021}. We set $j(0)=+\infty$. The grade of a subquotient is at least the grade of the original module, by the corresponding inclusion of characteristic varieties. For a nonzero relative holonomic module, Lemma~\ref{lem:relative-support} and~\eqref{eq:grade-dimension} give
\begin{equation}\label{eq:grade-parameter-support}
j(\Mmod)=n+\operatorname{codim}_{\Spec R}\Supp_R(\Mmod).
\end{equation}

Let $\omega_X$ be the canonical bundle. Relative duality is defined by
\[
\mathbb D_{X,R}(\Mmod)
=R\mathcal Hom_{\Amod_R}(\Mmod,\Amod_R)
 \otimes_{\mathcal O_X}\omega_X^{-1}[n].
\]
An $n$-Cohen--Macaulay module has dual concentrated in degree zero.

For a morphism $g:X\to Y$ of smooth varieties, set $\Amod_{X,R}=\Dmod_X\otimes_{\Cring}R$ and define the transfer bimodule
\[
\Amod^R_{Y\leftarrow X}
=\omega_X\otimes_{\mathcal O_X}
 \bigl(\mathcal O_X\otimes_{g^{-1}\mathcal O_Y}g^{-1}\Amod_{Y,R}\bigr)
 \otimes_{g^{-1}\mathcal O_Y}g^{-1}\omega_Y^{-1}.
\]
The relative direct image is
\[
g_+\Mmod
=Rg_*\bigl(\Amod^R_{Y\leftarrow X}
                 \otimes^L_{\Amod_{X,R}}\Mmod\bigr).
\]

\begin{theorem}[{\cite[Theorem 2.7]{Wu2022}}]\label{thm:relative-duality-direct-image}
If $g:X\to Y$ is proper, then relative direct image commutes with duality:
\[
g_+\mathbb D_{X,R}(\Mmod)
\simeq\mathbb D_{Y,R}(g_+\Mmod)
\]
for coherent relative $\Dmod$-modules $\Mmod$.
\end{theorem}

For a closed embedding $i:X\hookrightarrow Y$, relative Kashiwara equivalence identifies coherent $\Amod_{X,R}$-modules with coherent $\Amod_{Y,R}$-modules supported on $X\times\Spec R$; see \cite[Theorem 1.5]{FS19}. In particular, $i_+$ is exact, faithful, and $R$-linear, preserves relative holonomicity, and commutes with specialization of the parameters. We will use these properties in Section~\ref{sec:specialization}.

\subsection{$\Dmod_X[\svect]$ modules}\label{Dxmod}
Return to the tuple $F$ of the introduction, and set $R=\Cring[\svect]$ and $j:U=X\setminus D\hookrightarrow X$. The free $j_*\mathcal O_U[\svect]$-module generated by $F^{\svect}$ carries a left $\Dmod_X[\svect]$-action determined by
\begin{equation}\label{eq:Fs-action}
v(F^{\svect})=\left(\sum_{i=1}^r s_i\frac{v(f_i)}{f_i}\right)F^{\svect}
\end{equation}
for a local vector field $v$.

For $\mathbf a,\mathbf b\in\mathbb Z^r$, inequalities are understood coordinatewise. If $\mathbf a\leq\mathbf b$, define
\[
M_F^{\mathbf a,\mathbf b}
=\frac{\Dmod_X[\svect]F^{\svect+\mathbf a}}
        {\Dmod_X[\svect]F^{\svect+\mathbf b}},\qquad
B_F^{\mathbf a,\mathbf b}=\Ann_R M_F^{\mathbf a,\mathbf b}.
\]
We write $B_F^{\mathbf b}=B_F^{\mathbf0,\mathbf b}$ and
\[
Q_F=M_F^{\mathbf0,\one},\qquad B_F^{\one}=\Ann_R Q_F.
\]
The notation $\mathbf a<\mathbf b$ means $\mathbf a\leq\mathbf b$ and $\mathbf a\neq\mathbf b$.

\Needspace{6\baselineskip}
\begin{theorem}[{\cite[Theorem 3.2]{Wu2022}}]\label{thm:relative-properties-Fs}
The following properties hold.
\begin{enumerate}
\item For every $\mathbf a\in\mathbb Z^r$, the module $\Dmod_X[\svect]F^{\svect+\mathbf a}$ is relative holonomic and $n$-pure, and
\[
\CCrel\bigl(\Dmod_X[\svect]F^{\svect+\mathbf a}\bigr)
=\CC(j_*\mathcal O_U)\times\Cring^r.
\]
\item For $\mathbf a\leq\mathbf b$, the module $M_F^{\mathbf a,\mathbf b}$ is relative holonomic, and
\[
\Supp_R M_F^{\mathbf a,\mathbf b}
=p_2\bigl(\Chrel(M_F^{\mathbf a,\mathbf b})\bigr)
=Z(B_F^{\mathbf a,\mathbf b}).
\]
\item If $\mathbf a<\mathbf b$ and $F^{\mathbf b-\mathbf a}$ is not invertible, then
\[
\dim\Chrel(M_F^{\mathbf a,\mathbf b})=n+r-1,
\qquad j(M_F^{\mathbf a,\mathbf b})=n+1.
\]
\end{enumerate}
\end{theorem}

For $\boldsymbol\beta\in\mathbb Z^r$, put
\[
\tau_{\boldsymbol\beta}(h)(\svect)=h(\svect+\boldsymbol\beta).
\]
Substitution of the parameters gives
\begin{equation}\label{eq:parameter-translation}
B_F^{\mathbf a+\boldsymbol\beta,\mathbf b+\boldsymbol\beta}
=\tau_{\boldsymbol\beta}(B_F^{\mathbf a,\mathbf b}),\qquad
Z\bigl(\tau_{\boldsymbol\beta}(B_F^{\mathbf a,\mathbf b})\bigr)
=Z(B_F^{\mathbf a,\mathbf b})-\boldsymbol\beta.
\end{equation}
This convention will be used for all translations of ideals.

\begin{theorem}[Sabbah--Gyoja; {\cite[Theorem 3.4]{Wu2022}}]\label{thm:Sabbah-Gyoja}
The ideal $B_F^{\one}$ contains a nonzero polynomial which is a finite product of affine linear forms $L\cdot\svect+\alpha$, with $L\in\mathbb Z_{\geq0}^r\setminus\{0\}$ and $\alpha\in\mathbb Q$. Moreover, the polynomial may be chosen with $\alpha>0$ for every factor.
\end{theorem}

The assertion with rational constant terms is due to Sabbah; positivity is due to Gyoja. Constant factors, when present, are omitted.

\subsection{Localized maximal and minimal extensions}\label{sec:localized-extensions}
The maximal extension $j_*(\mathcal O_U[\svect]F^{\svect})$ need not be coherent over $\Dmod_X[\svect]$. Coherence is restored by localization in the parameter space.

\begin{theorem}[{\cite[Theorem 3.12]{Wu2022}}]\label{thm:maximal-extension}
For every prime ideal $\mathfrak p\subset R$, the module $j_*(\mathcal O_U[\svect]_{\mathfrak p}F^{\svect})$ is $n$-Cohen--Macaulay, and
\[
j_*(\mathcal O_U[\svect]_{\mathfrak p}F^{\svect})
=\Dmod_X[\svect]_{\mathfrak p}F^{\svect-k\one}
\qquad(k\gg0).
\]
\end{theorem}

The minimal extension is defined by duality:
\[
j_!(\mathcal O_U[\svect]_{\mathfrak p}F^{\svect})
:=\mathbb D_{X,R_{\mathfrak p}}j_*
       \mathbb D_{U,R_{\mathfrak p}}
       (\mathcal O_U[\svect]_{\mathfrak p}F^{\svect}).
\]
$j_!(\mathcal O_U[\svect]_{\mathfrak p}F^{\svect})$ is a sheaf(instead of a complex) and it is $n$-Cohen--Macaulay,see \cite[\S3.3]{Wu2022}.

\begin{theorem}[{\cite[Theorem 3.13]{Wu2022}}]\label{thm:minimal-extension}
For every prime ideal $\mathfrak p\subset R$, the minimal extension is $n$-Cohen--Macaulay. The natural morphism
\[
j_!(\mathcal O_U[\svect]_{\mathfrak p}F^{\svect})
\longrightarrow j_*(\mathcal O_U[\svect]_{\mathfrak p}F^{\svect})
\]
is injective, and
\[
j_!(\mathcal O_U[\svect]_{\mathfrak p}F^{\svect})
=\Dmod_X[\svect]_{\mathfrak p}F^{\svect+k\one}
\qquad(k\gg0).
\]
\end{theorem}

We therefore set
\begin{equation}\label{eq:localized-Psi}
\Psi_{F,\mathfrak p}(\mathcal O_X)
:=\frac{j_*(\mathcal O_U[\svect]_{\mathfrak p}F^{\svect})}
        {j_!(\mathcal O_U[\svect]_{\mathfrak p}F^{\svect})}
=\bigl(M_F^{-k\one,k\one}\bigr)_{\mathfrak p}
\qquad(k\gg0).
\end{equation}
The results in this subsection are local on $X$ and hence also apply on the smooth, possibly nonaffine, variety $Y$ of a log resolution.

\subsection{Left and right modules}\label{sec:side-changing}
Tensoring with $\omega_X$ gives an exact equivalence from left to right $\Amod_R$-modules:
\[
\Mmod\longmapsto\Mmod^r:=\Mmod\otimes_{\mathcal O_X}\omega_X.
\]
For a regular function $g$, a vector field $v$, and a local top form $\omega$, the action is
\[
(m\otimes\omega)g=gm\otimes\omega,\qquad
(m\otimes\omega)v=-(vm)\otimes\omega-m\otimes\mathcal L_v\omega.
\]
The coefficient ring $R$ acts centrally. The inverse equivalence sends a right module $\mathcal N$ to $\mathcal N\otimes_{\mathcal O_X}\omega_X^{-1}$, with
\[
v(n\otimes\eta)=-(nv)\otimes\eta+n\otimes\mathcal L_v\eta.
\]
These equivalences commute with localization in $R$. In local coordinates $x_1,\ldots,x_n$, with $dx=dx_1\wedge\cdots\wedge dx_n$, they identify
\[
\Dmod_X[\svect]F^{\svect}\otimes_{\mathcal O_X}\omega_X
=F^{\svect}dx\cdot\Dmod_X[\svect].
\]
\section{Log-resolution hyperplanes and the lower bound}\label{lower}

In this section we prove Theorem~\ref{Theorem:1.1}. The argument has three steps: a local normal-crossings calculation identifies a cyclic lattice with the minimal extension on the resolution; proper direct image places the original generator in the minimal extension downstairs; localization at a height-one prime then excludes all other codimension-one components.

Retain the log resolution $\mu:Y\to X$ and the numerical data~\eqref{eq:intro-resolution-data}. Put $F_Y=F\circ\mu$, and identify $Y\setminus\mu^{-1}D$ with $U=X\setminus D$. Denote the open embeddings by $j:U\hookrightarrow X$ and $j':U\hookrightarrow Y$. We work with right relative $\Dmod$-modules.

The argument is local on $X$. After passing to an affine open subset with local coordinates $x_1,\ldots,x_n$, we may trivialize $\omega_X$ by $dx=dx_1\wedge\cdots\wedge dx_n$. The assertions obtained on such a cover give the corresponding sheaf-theoretic assertions on $X$.

\subsection{The cyclic lattice on the resolution}

\begin{proposition}\label{prop:local-normal-crossing}
Let $\mathfrak p\subset R$ be a height-one prime such that
\begin{equation}\label{eq:avoid-resolution-hyperplanes}
L_E(\svect)+k_E+c\notin\mathfrak p
\qquad(E\in\mathcal E,\ c\in\mathbb Z_{>0}).
\end{equation}
Then, for every $\boldsymbol\beta\in\mathbb Z_{\geq0}^r$ and every standard basis vector $\mathbf e_i$,
\[
F_Y^{\svect+\boldsymbol\beta}\mu^*(dx)\cdot\Dmod_Y[\svect]_{\mathfrak p}
=F_Y^{\svect+\boldsymbol\beta+\mathbf e_i}\mu^*(dx)\cdot\Dmod_Y[\svect]_{\mathfrak p}.
\]
\end{proposition}

\begin{proof}
Choose local coordinates on $Y$. Locally,
\[
f_i\circ\mu=u_i\prod_E y_E^{N_{E,i}},\qquad
\mu^*(dx)=v\prod_E y_E^{k_E}\,dy,
\]
here the $u_j$ and $v$ are units.  Thus we need to prove $(v\prod_{j=1}^r u_j^{s_j+\beta_j}\prod_Ey_E^{L_E(\svect+\boldsymbol\beta)+k_E})\Dmod_Y[s]_\mathfrak p=(vu_i\prod_{j=1}^r u_j^{s_j+\beta_j}\prod_Ey_E^{L_E(\svect+\boldsymbol\beta+e_i)+k_E})\Dmod_Y[s]_\mathfrak p$. Set
$U^{\svect}=\prod_{j=1}^r u_j^{s_j}$.
Consider the $R$-algebra automorphism
\[
\Phi:\Dmod_Y[\svect]\longrightarrow\Dmod_Y[\svect]
\]
which fixes $\mathcal O_Y[\svect]$ pointwise and is given by
\[
\Phi(\partial_{y_a})
=
\partial_{y_a}
+\sum_{j=1}^r s_j\frac{\partial_{y_a}(u_j)}{u_j}.
\]
The
inverse is obtained by changing the plus sign to a minus sign.
Formally, $\Phi(P)=U^{-\svect}PU^{\svect}$.
Since $\Phi$ fixes $R$ pointwise, it extends to an automorphism
of $\Dmod_Y[\svect]_{\mathfrak p}$.

For the right $\Dmod$-module structure, multiplication by the
formal symbol $U^{\svect}$ satisfies
\[
(U^{\svect}m)\cdot\Phi(P)
=
U^{\svect}(m\cdot P),
\]
where $m$ is a local section of the corresponding monomial
right module. Consequently,
\[
(U^{\svect}m)\cdot\Dmod_Y[\svect]_{\mathfrak p}
=
U^{\svect}
\bigl(m\cdot\Dmod_Y[\svect]_{\mathfrak p}\bigr).
\]
Thus the desired equality of cyclic lattices is equivalent
to the corresponding equality with the common factor
$U^{\svect}$ removed. The remaining factors
$v\prod_j u_j^{\beta_j}$ and
$v u_i\prod_j u_j^{\beta_j}$ are units in $\mathcal O_Y$
and do not change the cyclic modules generated.
It therefore suffices to carry out the calculation in the
monomial case.

Multiplication by $f_i\circ\mu$ gives the inclusion from right to left. For the converse, fix $E$ and put
\[
m=N_{E,i},\qquad A_E=L_E(\svect)+L_E(\boldsymbol\beta)+k_E.
\]
If $m=0$, there is nothing to check in this coordinate. Otherwise, the right action gives
\[
(y_E^{A_E+m}dy)\cdot\partial_{y_E}^{m}
=(-1)^m\prod_{h=1}^{m}(A_E+h)y_E^{A_E}\,dy.
\]
Each factor has the form $L_E(\svect)+k_E+c$ with $c=L_E(\boldsymbol\beta)+h>0$, and is therefore invertible in $R_{\mathfrak p}$ by~\eqref{eq:avoid-resolution-hyperplanes}. Applying this calculation in every divisor coordinate gives the reverse inclusion.
\end{proof}

\begin{lemma}\label{lem:upstairs-jshriek}
Under the assumptions of Proposition~\ref{prop:local-normal-crossing},
\[
F_Y^{\svect}\mu^*(dx)\cdot\Dmod_Y[\svect]_{\mathfrak p}
=j'_!(\mathcal O_U[\svect]_{\mathfrak p}F_Y^{\svect})
  \otimes_{\mathcal O_Y}\omega_Y.
\]
\end{lemma}

\begin{proof}
Proposition~\ref{prop:local-normal-crossing}, applied along a path of unit shifts, gives
\[
F_Y^{\svect}\mu^*(dx)\cdot\Dmod_Y[\svect]_{\mathfrak p}
=F_Y^{\svect+k\one}\mu^*(dx)\cdot\Dmod_Y[\svect]_{\mathfrak p}
\qquad(k\geq0).
\]
It remains to identify the right-hand side for large $k$.

Locally write $\mu^*(dx)=g\,dy$, where $g$ is a unit times $\prod_Ey_E^{k_E}$. Choose $\ell\geq0$ such that $\ell L_E(\one)\geq k_E$ for every component meeting the neighborhood. Then $(f\circ\mu)^\ell/g$ is regular, and
\[
\begin{aligned}
F_Y^{\svect+(k+\ell)\one}dy\cdot\Dmod_Y[\svect]_{\mathfrak p}
&\subseteq F_Y^{\svect+k\one}g\,dy\cdot\Dmod_Y[\svect]_{\mathfrak p}\\
&\subseteq F_Y^{\svect+k\one}dy\cdot\Dmod_Y[\svect]_{\mathfrak p}.
\end{aligned}
\]
For $k\gg0$, both outer terms are the right-module form of the minimal extension, by Theorem~\ref{thm:minimal-extension}. Hence the middle term agrees with them. This proves the assertion locally, and the equalities glue on $Y$.
\end{proof}

\subsection{Descent of the generator}

\begin{lemma}\label{lem:global-section-trick}
Under the same assumptions,
\[
F^{\svect}dx\in
j_!(\mathcal O_U[\svect]_{\mathfrak p}F^{\svect})
\otimes_{\mathcal O_X}\omega_X.
\]
Consequently, for every $k\geq0$,
\[
F^{\svect}dx\cdot\Dmod_X[\svect]_{\mathfrak p}
=F^{\svect+k\one}dx\cdot\Dmod_X[\svect]_{\mathfrak p}.
\]
\end{lemma}

\begin{proof}
Write $\mathcal N=F_Y^{\svect}\mu^*(dx)\cdot\Dmod_Y[\svect]_{\mathfrak p}$. By Lemma~\ref{lem:upstairs-jshriek}, properness of $\mu$, and compatibility of direct image with duality,
\begin{equation}\label{eq:descent-minimal-extension}
\mu_+\mathcal N
\simeq j_!(\mathcal O_U[\svect]_{\mathfrak p}F^{\svect})
       \otimes_{\mathcal O_X}\omega_X.
\end{equation}
Here we use $\mu\circ j'=j$. In particular, this direct image is concentrated in degree zero.

We identify a section of~\eqref{eq:descent-minimal-extension} by tracking the generator. It defines a right-module morphism
\[
\Dmod_Y[\svect]_{\mathfrak p}\longrightarrow\mathcal N,
\qquad P\longmapsto F_Y^{\svect}\mu^*(dx)\cdot P.
\]
Let
\[
\Dmod_{Y\to X}[\svect]_{\mathfrak p}
=\mathcal O_Y\otimes_{\mu^{-1}\mathcal O_X}
                  \mu^{-1}\Dmod_X[\svect]_{\mathfrak p}
\]
be the right-module transfer bimodule. Derived tensor product with this bimodule, preceded by the map $\mathcal O_Y\to\Dmod_{Y\to X}[\svect]_{\mathfrak p}$ sending $1$ to $1\otimes1$, gives
\[
\mathcal O_Y\longrightarrow
\mathcal N\otimes^L_{\Dmod_Y[\svect]_{\mathfrak p}}
                   \Dmod_{Y\to X}[\svect]_{\mathfrak p}.
\]
Since $\mu$ is a proper birational morphism between smooth varieties, $R\mu_*\mathcal O_Y\simeq\mathcal O_X$. Applying $R\mu_*$, the image of $1$ defines a section $u$ of $\mu_+\mathcal N$. Over $U$, where $\mu$ is an isomorphism, it restricts to $F^{\svect}dx$.

By Theorem~\ref{thm:minimal-extension}, the right-hand side of~\eqref{eq:descent-minimal-extension} embeds in the maximal extension. A section of the latter is determined by its restriction to $U$. Thus $u=F^{\svect}dx$, proving the first assertion.

For $k\gg0$, Theorem~\ref{thm:minimal-extension} now gives
\[
F^{\svect}dx\in F^{\svect+k\one}dx\cdot\Dmod_X[\svect]_{\mathfrak p}.
\]
The opposite inclusion of cyclic modules follows by multiplication by $f^k$. Equality for every $k\geq0$ follows by placing each intermediate lattice between the equal terms for $0$ and for a sufficiently large shift.
\end{proof}

\subsection{Codimension-one components}

\begin{theorem}\label{thm:codim-one-components}
For $\mathbf a\in\mathbb Z_{\geq0}^r$, every irreducible component of $Z(B_F^{\mathbf a})$ of codimension one is
\[
Z\bigl(L_E(\svect)+k_E+c\bigr)
\]
for some $E\in\mathcal E$ and $c\in\mathbb Z_{>0}$.
\end{theorem}

\begin{proof}
Let $H$ be such a component and $\mathfrak p=I(H)$. Suppose that $H$ is not one of the stated hyperplanes. Since $\mathfrak p$ has height one and each $L_E(\svect)+k_E+c$ is a nonconstant linear form, condition~\eqref{eq:avoid-resolution-hyperplanes} holds.

Choose $k$ with $\mathbf a\leq k\one$. Lemma~\ref{lem:global-section-trick} identifies the outer terms in
\[
\begin{aligned}
F^{\svect}dx\cdot\Dmod_X[\svect]_{\mathfrak p}
&\supseteq F^{\svect+\mathbf a}dx\cdot\Dmod_X[\svect]_{\mathfrak p}\\
&\supseteq F^{\svect+k\one}dx\cdot\Dmod_X[\svect]_{\mathfrak p}.
\end{aligned}
\]
Hence all three terms agree. After changing sides, this says
\[
\bigl(M_F^{\mathbf0,\mathbf a}\bigr)_{\mathfrak p}=0.
\]
On the other hand, this module is relative holonomic, and Lemma~\ref{lem:relative-support} identifies its parameter support with $Z(B_F^{\mathbf a})$. The generic point $\mathfrak p$ of $H$ belongs to that support, a contradiction.
\end{proof}

This proves Theorem~\ref{Theorem:1.1}. The argument gives the positivity of $c$ directly from the invertible factors in the normal-crossings calculation.
\section{Diagonal specialization and positive translations}\label{sec:specialization}

We first extend van der Veer's fiber nonvanishing and diagonal specialization theorems \cite[Theorems E and F]{vanderVeer2021} from affine space to a smooth affine variety. Relative Kashiwara equivalence gives the first extension, and the kernel calculation from \cite[\S5]{vanderVeer2021} gives the second. We then prove the positive translation theorem and determine the smallest interval containing the diagonal slice.
Throughout this section, $X$ is smooth and affine and $R=\Cring[s_1,\ldots,s_r]$.

\subsection{Nonvanishing of fibers}

\begin{theorem}\label{thm:fiber-nonvanishing}
Let $\Mmod$ be a relative holonomic $\Dmod_X[\svect]$-module. For every $\alpha\in Z(\Ann_R\Mmod)$, with corresponding maximal ideal $\mathfrak m_\alpha\subset R$,
\[
\Mmod\otimes_R R/\mathfrak m_\alpha\neq0.
\]
\end{theorem}

\begin{proof}
Choose a closed embedding $i:X\hookrightarrow\mathbb A^N$. By relative Kashiwara equivalence, $i_+\Mmod$ is relative holonomic. Since $i_+$ is faithful and $R$-linear,
\[
\Ann_R(i_+\Mmod)=\Ann_R\Mmod.
\]
Moreover, $(i_+\Mmod)\otimes_R R/\mathfrak m_\alpha
\simeq i_+\bigl(\Mmod\otimes_R R/\mathfrak m_\alpha\bigr).$

Theorem E of \cite{vanderVeer2021}, applied on $\mathbb A^N$, says that the left-hand side is nonzero. Therefore the fiber of $\Mmod$ is nonzero.
\end{proof}

\subsection{The specialization map and its kernel}
Let $F=(f_1,\ldots,f_r)$ and $f=\prod_i f_i$ be as before. Put
\[
\Delta=\{s_1=\cdots=s_r\},\qquad
I_\Delta=(s_1-s_r,\ldots,s_{r-1}-s_r),
\]
so that $R/I_\Delta=\Cring[s]$. We identify $s\in\Cring$ with $s\one\in\Delta$ and write $\nabla(\alpha)=\alpha+\one$ for translation on the parameter space.

Consider the modules
\[
\begin{aligned}
N_F&=\mathcal O_X[f^{-1},\svect]F^{\svect},&
M_F&=\Dmod_X[\svect]F^{\svect},\\
N_f&=\mathcal O_X[f^{-1},s]f^s,&
M_f&=\Dmod_X[s]f^s.
\end{aligned}
\]
The $\Dmod_X$-linear shift maps are
\[
\begin{aligned}
\nabla_F\bigl(h(\svect)F^{\svect}\bigr)
 &=h(\svect+\one)F^{\svect+\one},\\
\nabla_f\bigl(q(s)f^s\bigr)&=q(s+1)f^{s+1}.
\end{aligned}
\]
They are semilinear over the parameter rings, with images
\[
\nabla_F(M_F)=\Dmod_X[\svect]F^{\svect+\one},\qquad
\nabla_f(M_f)=\Dmod_X[s]f^{s+1}.
\]
Thus $Q_F=M_F/\nabla_F(M_F)$ and $Q_f=M_f/\nabla_f(M_f)$. Since $X$ is affine, we freely pass between these quasi-coherent sheaves and their modules of global sections when choosing generators.

\begin{lemma}\label{lem:diagonal-exact-sequence}
The map
\[
\eta_F:N_F\longrightarrow N_f,\qquad
h(\svect)F^{\svect}\longmapsto h(s,\ldots,s)f^s
\]
is surjective and $\Dmod_X$-linear. Its kernel is $I_\Delta N_F$, and
\[
\eta_F(M_F)=M_f,\qquad
\eta_F(\nabla_F(M_F))=\nabla_f(M_f).
\]
Consequently, there is a short exact sequence
\begin{equation}\label{eq:diagonal-exact-sequence}
0\longrightarrow K\longrightarrow Q_F/I_\Delta Q_F
\longrightarrow Q_f\longrightarrow0,
\end{equation}
where
\[
K=\frac{(\ker\eta_F\cap M_F)+\nabla_F(M_F)}
        {I_\Delta M_F+\nabla_F(M_F)}.
\]
\end{lemma}

\begin{proof}
Surjectivity and the description of the kernel follow from substitution of the parameters. For a vector field $v$, formula~\eqref{eq:Fs-action} specializes to
\[
s\left(\sum_i\frac{v(f_i)}{f_i}\right)f^s
=s\frac{v(f)}{f}f^s=v(f^s),
\]
which proves $\Dmod_X$-linearity. The two identities for the images of the cyclic submodules follow from their generators. The induced map $Q_F\to Q_f$ has kernel
\[
\frac{(\ker\eta_F\cap M_F)+\nabla_F(M_F)}{\nabla_F(M_F)}.
\]
Since $I_\Delta M_F\subseteq\ker\eta_F\cap M_F$, the map factors through $Q_F/I_\Delta Q_F$, with kernel $K$ as stated.
\end{proof}

\begin{lemma}\label{lem:diagonal-kernel-support}
With the preceding notation,
\[
\Supp_R K\subseteq\bigcup_{m>0}\nabla^m\bigl(Z(B_F^{\one})\bigr).
\]
\end{lemma}

\begin{proof}
The module $K$ is coherent over $\Dmod_X[\svect]$. Choose finitely many global generators $\overline T_1,\ldots,\overline T_q$. Since $\ker\eta_F=I_\Delta N_F$, after clearing powers of $f$ we may represent them by
\[
T_\nu=h_\nu(\svect)f^{-d}F^{\svect}+\nabla_F(W_\nu),
\qquad
h_\nu(\svect)\in I_\Delta\mathcal O_X(X)[\svect],
\]
with a common integer $d\geq1$ and $W_\nu\in M_F$.

For arbitrary nonzero polynomials $b_j\in B_F^{\one}$, $1\leq j\leq d$, choose functional equations
\[
b_j(\svect)F^{\svect}=P_j(\svect)F^{\svect+\one}.
\]
Substitution of $\svect-j\one$ gives
\[
b_j(\svect-j\one)f^{-j}F^{\svect}
=P_j(\svect-j\one)f^{-j+1}F^{\svect}.
\]
Iterating these equations shows that, for
\[
c(\svect)=\prod_{j=1}^{d}b_j(\svect-j\one),
\]
one has $c(\svect)f^{-d}F^{\svect}\in M_F$. It follows that
\[
c(\svect)h_\nu(\svect)f^{-d}F^{\svect}\in I_\Delta M_F.
\]
Also,
\[
c(\svect)\nabla_F(W_\nu)
=\nabla_F\bigl(c(\svect-\one)W_\nu\bigr)\in\nabla_F(M_F).
\]
Thus $c(\svect)$ annihilates every generator of $K$, and hence annihilates $K$.

Since the $b_j$ were arbitrary, the product ideal
\[
\prod_{j=1}^{d}\tau_{-j\one}(B_F^{\one})
\]
annihilates $K$. Its zero locus is
\[
\bigcup_{j=1}^{d}\bigl(Z(B_F^{\one})+j\one\bigr),
\]
which proves the asserted inclusion.
\end{proof}

\subsection{The diagonal specialization theorem}

\begin{theorem}\label{thm:diagonal-specialization}
Let $F$ be a tuple of nonzero regular functions on a smooth complex affine variety, and assume that $f=\prod_i f_i$ is not invertible. Then
\begin{equation}\label{eq:diagonal-specialization}
\bigl(Z(B_F^{\one})\cap\Delta\bigr)
\setminus\bigcup_{m>0}\nabla^m\bigl(Z(B_F^{\one})\bigr)
\subseteq Z(b_f),
\end{equation}
where $Z(b_f)$ is identified with a subset of $\Delta$.
\end{theorem}

\begin{proof}
Let $\alpha=s_0\one$ belong to the left-hand side of~\eqref{eq:diagonal-specialization}. Lemma~\ref{lem:diagonal-kernel-support} gives $K_{\mathfrak m_\alpha}=0$. Localizing~\eqref{eq:diagonal-exact-sequence} at $\mathfrak m_\alpha$ and then taking the residue-field fiber yields
\[
Q_F\otimes_R R/\mathfrak m_\alpha
\simeq Q_f\otimes_{\Cring[s]}\Cring[s]/(s-s_0).
\]
By Theorem~\ref{thm:relative-properties-Fs}, $Q_F$ is relative holonomic. Since $\alpha\in Z(\Ann_R Q_F)$, Theorem~\ref{thm:fiber-nonvanishing} makes the left-hand side nonzero. The polynomial $b_f(s)$ annihilates $Q_f$ and acts on the right-hand fiber as the scalar $b_f(s_0)$. Therefore $b_f(s_0)=0$.
\end{proof}

\subsection{The diagonal slice}\label{sec:diagonal-slice}
With $\Delta$ and $\nabla$ as above, put
\[
S=\{t\in\Cring:t\one\in Z(B_F^{\one})\}.
\]

\begin{lemma}\label{lemma:1}
The set $S$ is a finite subset of $\mathbb Q_{<0}$.
\end{lemma}

\begin{proof}
By Theorem~\ref{thm:Sabbah-Gyoja}, there is a nonzero polynomial
\[
b(\svect)=\prod_{(L,\alpha)}(L\cdot\svect+\alpha)\in B_F^{\one},
\]
with $L\in\mathbb Z_{\geq0}^r\setminus\{0\}$ and $\alpha\in\mathbb Q_{>0}$. Its restriction to the diagonal is
\[
b(t\one)=\prod_{(L,\alpha)}\bigl((L\cdot\one)t+\alpha\bigr).
\]
This is a nonzero polynomial whose roots are negative rational numbers, and $S$ is contained in its zero locus.
\end{proof}

\begin{proposition}\label{Proposition:1}
For every $t\in S$, there are $\rho\in Z(b_f)$ and $M\in\mathbb Z_{\geq0}$ such that $t=\rho+M$. In particular,
\begin{equation}\label{eq:diagonal-slice-bound}
t\geq-n+\delta_f\qquad(t\in S).
\end{equation}
\end{proposition}

\begin{proof}
Choose the least element $\rho$ of the finite nonempty set
\[
S\cap(t-\mathbb Z_{\geq0}).
\]
If $\rho\notin Z(b_f)$, Theorem~\ref{thm:diagonal-specialization} implies that $\rho\one\in\nabla^m(Z(B_F^{\one}))$ for some $m>0$. Hence $(\rho-m)\one\in Z(B_F^{\one})$, so $\rho-m$ belongs to the same set and is smaller than $\rho$, a contradiction. Thus $\rho\in Z(b_f)$ and $t-\rho\in\mathbb Z_{\geq0}$. The bound follows from~\eqref{eq:intro-root-bound}.
\end{proof}
\subsection{Positive translations}\label{sec:positive-translations}
We now prove the refinement of Theorem~\ref{thmbvdwz}(2) stated in the introduction. For $\mathbf a\in\mathbb Z_{\geq0}^r$, write
\[
W_{\mathbf a}=Z(B_F^{\mathbf a}),\qquad
W_{\mathbf a}^{(1)}=\bigcup_{\substack{H\text{ an irreducible component of }W_{\mathbf a}\\
                                    \operatorname{codim}H=1}}H.
\]

\begin{proposition}[Positive translation]\label{prop:positive-translation}
Assume that $F^{\mathbf a}$ is not invertible. There exists an integer $q\geq0$ such that
\begin{equation}\label{eq:positive-cover}
W_{\mathbf a}\subseteq\bigcup_{k=0}^{q}
\bigl(W_{\mathbf a}^{(1)}-k\mathbf a\bigr).
\end{equation}
Consequently, every irreducible component $C$ of $W_{\mathbf a}$ of codimension at least two satisfies $C+k\mathbf a\subseteq H$ for some integer $k\geq1$ and some codimension-one component $H$ of $W_{\mathbf a}$.
\end{proposition}

\begin{proof}
We work with global sections on the affine variety $X$. Set
\[
N=\Dmod_X[\svect]F^{\svect},\qquad
\mathscr E=\mathcal O_X[f^{-1}][\svect]F^{\svect}.
\]
The map
\[
\tau\bigl(g(x,\svect)F^{\svect}\bigr)
=g(x,\svect+\mathbf a)F^{\mathbf a}F^{\svect}
\]
is an invertible semilinear map on $\mathscr E$. It commutes with the $\Dmod_X$-action and satisfies $\tau N\subseteq N$.

Let $\widehat L\subseteq\mathscr E$ be the full maximal tame pure extension of $N$. By \cite[\S4.4]{BvdVWZ21b}, it is the largest finitely generated $\Dmod_X[\svect]$-submodule containing $N$ for which
\[
j(\widehat L/N)\geq n+2.
\]
Since $\widehat L$ is finitely generated, it lies in $\Dmod_X[\svect]F^{\svect-\ell\one}$ for some $\ell\geq0$. Thus $\widehat L$ and its subquotients are relative holonomic. By \cite[Lemma 4.4.1]{BvdVWZ21b}, we have $\tau\widehat L\subseteq\widehat L$, and the quotient
\[
P=\widehat L/\tau\widehat L
\]
is relative holonomic and $(n+1)$-pure. Moreover, $\Supp_R P$ has pure codimension one by \cite[Lemma 4.4.2(iii)]{BvdVWZ21b}.

We first identify this support. The natural map $M=N/\tau N\to P$ has kernel and cokernel
\[
K_{\mathbf a}=\frac{N\cap\tau\widehat L}{\tau N},\qquad
C_{\mathbf a}=\frac{\widehat L}{N+\tau\widehat L}.
\]
The first is a submodule of $\tau\widehat L/\tau N$, and the second is a quotient of $\widehat L/N$. Translation of the parameters preserves grade, so
\[
j(K_{\mathbf a}),\ j(C_{\mathbf a})\geq n+2.
\]
Both modules are relative holonomic. By~\eqref{eq:grade-parameter-support}, their parameter supports have codimension at least two. Hence $M\to P$ is an isomorphism after localization at every height-one prime of $R$. Since $\Supp_R M=W_{\mathbf a}$ and every irreducible component of $\Supp_R P$ has codimension one, the two supports have exactly the same codimension-one components. Therefore
\begin{equation}\label{eq:pure-extension-divisorial-support}
\Supp_R P=W_{\mathbf a}^{(1)}.
\end{equation}

To obtain a finite covering in the required direction, put $Q=\widehat L/N$. The map $\tau$ induces a semilinear endomorphism $\overline\tau$ of $Q$. Each $\ker\overline\tau^{m}$ is a $\Dmod_X[\svect]$-submodule, and these kernels form an ascending chain. Since $Q$ is noetherian, there exists $q\geq0$ such that
\begin{equation}\label{eq:shift-kernel-stabilization}
\ker\overline\tau^{q+1}=\ker\overline\tau^{q}.
\end{equation}
If $z\in\widehat L$ satisfies $\tau^{q+1}z\in N$, then~\eqref{eq:shift-kernel-stabilization} gives $\tau^qz\in N$. Consequently,
\[
N\cap\tau^{q+1}\widehat L\subseteq\tau N.
\]
The module $N/(N\cap\tau^{q+1}\widehat L)$ therefore surjects onto $M$ and injects into $\widehat L/\tau^{q+1}\widehat L$. Exactness of localization gives
\[
W_{\mathbf a}\subseteq
\Supp_R\bigl(\widehat L/\tau^{q+1}\widehat L\bigr).
\]

The module on the right has a finite filtration with factors
\[
\tau^k\widehat L/\tau^{k+1}\widehat L,\qquad 0\leq k\leq q.
\]
An annihilator $b(\svect)$ of $P$ becomes $b(\svect+k\mathbf a)$ on the $k$th factor. Since $\tau^k$ is invertible on $\mathscr E$, this identifies the corresponding annihilator ideals by parameter translation. Thus~\eqref{eq:pure-extension-divisorial-support} gives
\[
\Supp_R\bigl(\tau^k\widehat L/\tau^{k+1}\widehat L\bigr)
=\Supp_R P-k\mathbf a
=W_{\mathbf a}^{(1)}-k\mathbf a.
\]
Taking the union of these supports proves~\eqref{eq:positive-cover}.

An irreducible component $C$ of $W_{\mathbf a}$ is contained in one member of this finite union, and hence in $H-k\mathbf a$ for some codimension-one component $H$ of $W_{\mathbf a}$. If $\operatorname{codim}C\geq2$, then $k=0$ would give $C\subseteq H$, contradicting maximality of $C$ as an irreducible component of $W_{\mathbf a}$. Thus $k\geq1$.
\end{proof}

Proposition~\ref{prop:positive-translation} proves Theorem~\ref{thm:intro-positive-translation}. 

\subsection{The diagonal interval}\label{sec:diagonal-interval}
Retain the finite set $S$ from Section~\ref{sec:diagonal-slice}, and put
\[
\rho_- = \min Z(b_f),\qquad \rho_+ = \max Z(b_f).
\]

\begin{theorem}[The diagonal interval]\label{thm:diagonal-interval}
One has
\begin{equation}\label{eq:diagonal-interval}
Z(b_f)\subseteq S\subseteq[\rho_-,\rho_+].
\end{equation}
Both endpoints belong to $S$. Equivalently,
\[
\min S=\rho_-,\qquad
\max S=\rho_+=-\operatorname{lct}_X(f).
\]
\end{theorem}

\begin{proof}
For every $b(\svect)\in B_F^{\one}$, diagonal specialization of its functional equation gives
\[
b(t\one)f^t\in\Dmod_X[t]f^{t+1}.
\]
Hence $b_f(t)$ divides $b(t\one)$. As this holds for every $b\in B_F^{\one}$, we obtain $Z(b_f)\subseteq S$.

For the lower bound, let $t\in S$. Proposition~\ref{Proposition:1} gives $t=\rho+m$ with $\rho\in Z(b_f)$ and $m\in\mathbb Z_{\geq0}$. Thus $t\geq\rho_-$.

For the upper bound, apply Proposition~\ref{prop:positive-translation} with $\mathbf a=\one$. There are $k\geq0$ and a codimension-one component $H$ of $Z(B_F^{\one})$ such that $(t+k)\one\in H$. Fix the log resolution from the introduction. By Theorem~\ref{Theorem:1.1},
\[
H=Z\bigl(L_E(\svect)+k_E+c\bigr),\qquad c\in\mathbb Z_{>0},
\]
for some $E\in\mathcal E$. Set $N_E=L_E(\one)>0$. Then
\begin{equation}\label{eq:diagonal-upper-endpoint}
t=-\frac{k_E+c}{L_E(\mathbf 1)}-k
\leq-\frac{k_E+1}{L_E(\mathbf 1)}
\leq-\operatorname{lct}_X(f)=\rho_+.
\end{equation}
Here we used the log-resolution formula
\[
\operatorname{lct}_X(f)
=\min_{E\in\mathcal E}\frac{k_E+1}{L_E(\one)}
\]
and the largest-root characterization of the log canonical threshold \cite[Theorem 2]{BudurMustataSaito2006}. This proves the interval inclusion. The inclusion $Z(b_f)\subseteq S$ gives $\rho_-,\rho_+\in S$, so both bounds are attained.
\end{proof}

Together with Proposition~\ref{Proposition:1} and  Lemma~\ref{lemma:1}, this proves Theorem~\ref{thm:intro-diagonal-interval}.
We now prove Theorem~\ref{Theorem:1.2} using the estimate for the diagonal slice obtained. A filtration by unit shifts transfers this estimate first to each coordinate shift and then to an arbitrary nonnegative integral shift.

\subsection{Filtrations by unit shifts}
The following decomposition is the lattice-path form of \cite[Proposition 4.7]{Budur2015BSlocal}.

\begin{lemma}\label{lem:lattice-path}
Let $\mathbf a\in\mathbb Z_{\geq0}^r$, and choose a monotone lattice path
\[
\mathbf0=\boldsymbol\beta^0,\boldsymbol\beta^1,\ldots,
\boldsymbol\beta^m=\mathbf a,\qquad
\boldsymbol\beta^{j+1}=\boldsymbol\beta^j+\mathbf e_{i_j}.
\]
Then
\begin{equation}\label{eq:lattice-path-support}
Z(B_F^{\mathbf a})
=\bigcup_{j=0}^{m-1}
Z\bigl(\tau_{\boldsymbol\beta^j}(B_F^{\mathbf e_{i_j}})\bigr)
=\bigcup_{j=0}^{m-1}
\bigl(Z(B_F^{\mathbf e_{i_j}})-\boldsymbol\beta^j\bigr).
\end{equation}
\end{lemma}

\begin{proof}
The path gives a descending chain of cyclic modules
\[
\Dmod_X[\svect]F^{\svect+\boldsymbol\beta^0}
\supseteq\cdots\supseteq
\Dmod_X[\svect]F^{\svect+\boldsymbol\beta^m}.
\]
Its successive quotients are $M_F^{\boldsymbol\beta^j,\boldsymbol\beta^{j+1}}$. Parameter support is the union of the supports of these quotients, by exactness of localization. Theorem~\ref{thm:relative-properties-Fs} identifies these supports with their Bernstein--Sato zero loci, and~\eqref{eq:parameter-translation} gives the stated translations.
\end{proof}

\begin{lemma}\label{lem:unit-shift-bound}
Let $H_0=Z(L(\svect)+d)$ be a codimension-one component of $Z(B_F^{\mathbf e_i})$, where $L(\svect)=\sum_j l_js_j$, $l_j\geq0$, and $L(\one)>0$. Then
\begin{equation}\label{equation:1}
d\leq L(\mathbf e_i)+(n-1-\delta_f)L(\one).
\end{equation}
\end{lemma}

\begin{proof}
Choose a path from $\mathbf0$ to $\one$ whose last step is $\mathbf e_i$. The starting point of that step is $\one-\mathbf e_i$. Lemma~\ref{lem:lattice-path} implies
\[
H_0-(\one-\mathbf e_i)\subseteq Z(B_F^{\one}).
\]
The diagonal meets this translated hyperplane at $t\one$, where
\[
t=-\frac{d+L(\one)-L(\mathbf e_i)}{L(\one)}.
\]
Since $t\in S$, Proposition~\ref{Proposition:1} gives $t\geq-n+\delta_f$. Rearranging proves~\eqref{equation:1}.
\end{proof}

\subsection{Proof of the main estimate}

\begin{proof}[Proof of Theorem~\ref{Theorem:1.2}]
Let $H=Z(L_E(\svect)+k_E+c)$ be the given component. Choose a lattice path from $\mathbf0$ to $\mathbf a$. By~\eqref{eq:lattice-path-support}, for some step $j$ the translate
\[
H_0=H+\boldsymbol\beta^j
=Z\bigl(L_E(\svect)+k_E+c-L_E(\boldsymbol\beta^j)\bigr)
\]
is a codimension-one component of $Z(B_F^{\mathbf e_{i_j}})$. Here a codimension-one member of the finite union must be a component of one of its terms; the zero loci in question are proper by the existence of Bernstein--Sato ideals.

Apply Lemma~\ref{lem:unit-shift-bound} with $L=L_E$ and $d=k_E+c-L_E(\boldsymbol\beta^j)$. It gives
\[
k_E+c-L_E(\boldsymbol\beta^j)
\leq N_{E,i_j}+(n-1-\delta_f)L_E(\one).
\]
Since $\boldsymbol\beta^j+\mathbf e_{i_j}\leq\mathbf a$ and the coefficients of $L_E$ are nonnegative,
\[
L_E(\boldsymbol\beta^j)+N_{E,i_j}\leq L_E(\mathbf a).
\]
Combining the last two inequalities yields
\[
c\leq L_E(\mathbf a)+(n-1-\delta_f)L_E(\one)-k_E.
\]
This proof applies to the stated representation of $H$; no choice of a new resolution component is needed after translation.
\end{proof}

\subsection{Examples and sharpness}\label{sec:sharpness}

\begin{example}[Monomials on a curve]\label{ex:curve}
Let $X=\mathbb A^1$ and $F=(x^{N_1},\ldots,x^{N_r})$, where $N_i\geq0$ and $\sum_iN_i>0$. Put $L(\svect)=\sum_iN_is_i$. A one-coordinate calculation gives
\[
B_F^{\mathbf a}
=\left(\prod_{h=1}^{L(\mathbf a)}(L(\svect)+h)\right),
\]
with the empty product interpreted as $1$. Indeed, differentiating $x^{L(\svect)+L(\mathbf a)}$ exactly $L(\mathbf a)$ times gives the displayed product, and the successive one-coordinate quotients show that every displayed linear factor must divide any annihilator. If $L(\mathbf a)>0$, the component $L(\svect)+L(\mathbf a)=0$ attains the upper bound: the identity is a log resolution, $k_E=0$, and $\delta_f=0$.
\end{example}

\begin{example}[A normal-crossings tuple]\label{ex:normal-crossings-bound}
Let $X=\mathbb A^n$, $n\geq2$, and $F=(x_1,\ldots,x_n)$. The coordinatewise calculation gives
\[
B_F^{\mathbf a}
=\left(\prod_{i=1}^n\prod_{h=1}^{a_i}(s_i+h)\right).
\]
Since $b_f(s)=(s+1)^n$, one has $\alpha_f=\delta_f=1$. For the divisor $E_i=(x_i=0)$, the theorem gives $c\leq a_i+n-2$, whereas the actual largest value is $a_i$ when $a_i>0$. Thus the bound is attained for $n=2$, including tuples with more than one function, but need not be attained when $n>2$.
\end{example}

\begin{remark}[The one-variable case]
When $r=1$, a path of length $a$ gives
\[
Z(B_f^a)=\bigcup_{j=0}^{a-1}\bigl(Z(b_f)-j\bigr).
\]
Hence the bound is attained whenever the least root of $b_f$ is $-n+\delta_f$. In particular, Saito's estimate is attained for quasihomogeneous isolated hypersurface singularities \cite[Theorem 0.4 and the following discussion]{Saito1994Microlocal}; for $n\geq2$, this gives the asserted sharpness in that class. The curve case is covered by Example~\ref{ex:curve}.
\end{remark}
\section{Specialization complexes and monodromy zeta functions}\label{monodromy}

The upper-bound argument is complete. We now return to the localized extensions of Section~\ref{sec:localized-extensions} and their relation to monodromy. Our aim is to express the local index comparison as an identity between divisor-valued constructible functions and Lagrangian cycles. This identity detects monodromy support; compatibility with proper pushforward gives the multivariable A'Campo formula.

\subsection{The specialization complex and its normalization}\label{sec:sabbah}
Set
\begin{equation}\label{eq:mono:normalization}
A=\Cring[t_1^{\pm1},\ldots,t_r^{\pm1}],\qquad
\mathbb T=\Spec A,
\end{equation}
and use the universal covering map
\[
\Exp:\Cring^r\longrightarrow\mathbb T,\qquad
(s_1,\ldots,s_r)\longmapsto
(e^{-2\pi\sqrt{-1}s_1},\ldots,e^{-2\pi\sqrt{-1}s_r}).
\]
The sheaf $\mathcal L_A=\Exp_!\Cring_{\Cring^r}$ is the universal rank-one local system of free $A$-modules on $\mathbb T$.

Let $i:D\hookrightarrow X$ and $j:U\hookrightarrow X$. For a constructible complex $\mathcal F^\bullet$ of $\Cring$-sheaves on $X$, Sabbah's specialization complex is
\[
\psi_F(\mathcal F^\bullet)
=i_*i^{-1}Rj_*\bigl(j^{-1}\mathcal F^\bullet
                 \otimes_{\Cring}(F|_U)^{-1}\mathcal L_A\bigr);
\]
see \cite{S90}. It is a constructible complex of $A$-modules supported on $D$. We write
\[
S(F)=\bigcup_{x\in D}\bigcup_j
\Supp_A\mathcal H^j\bigl(\psi_F(\Cring_X)\bigr)_x
\subseteq\mathbb T.
\]

\begin{theorem}[{\cite[Theorem 1.5.1]{BudurVanDerVeerWuZhou2021}}]\label{thm:support-Bernstein-Sato}
One has $\Exp(Z(B_F^{\one}))=S(F)$. Moreover, $S(F)$ is a finite union of torsion-translated subtori of codimension one in $\mathbb T$.
\end{theorem}

\subsection{Divisor-valued constructible functions}
Let $\mathcal F(X)$ be the group of integer-valued algebraically constructible functions on $X$, and let $\mathcal L(X)$ be the free abelian group on irreducible algebraic conic Lagrangian subvarieties of $T^*X$. Write $\Div(A)$ for the free abelian group on prime divisors of $\mathbb T$. Put
\begin{equation}\label{eq:mono:groups}
\mathcal F_A(X)=\mathcal F(X)\otimes_{\mathbb Z}\Div(A),\qquad
\mathcal L_A(X)=\mathcal L(X)\otimes_{\mathbb Z}\Div(A).
\end{equation}
An element $\sum_{\mathfrak q}\varphi_{\mathfrak q}\otimes[V(\mathfrak q)]$ of $\mathcal F_A(X)$ is viewed as a divisor-valued constructible function, with value $\sum_{\mathfrak q}\varphi_{\mathfrak q}(x)[V(\mathfrak q)]$ at $x$.

For an irreducible closed subvariety $Z\subseteq X$, let $T_Z^*X=\overline{T^*_{Z_{\mathrm{sm}}}X}$ and let $\operatorname{Eu}_Z$ be its local Euler obstruction function, extended by zero. We normalize the characteristic-cycle isomorphism by
\begin{equation}\label{eq:mono:cc-map}
\CC:\mathcal F(X)\xrightarrow{\sim}\mathcal L(X),\qquad
\CC(\operatorname{Eu}_Z)=(-1)^{\dim Z}[T_Z^*X];
\end{equation}
see \cite[Chapter IX]{KS13}. Tensoring with $\Div(A)$ gives
\[
\CC_A=\CC\otimes\operatorname{id}:
\mathcal F_A(X)\xrightarrow{\sim}\mathcal L_A(X).
\]

Let $\mathfrak P_F$ be the set of generic points of the irreducible components of $S(F)$. For a height-one prime $\mathfrak q\subset A$, define
\begin{equation}\label{eq:mono:length}
\varphi_{\mathfrak q}(x)
=\sum_j(-1)^j\ell_{A_{\mathfrak q}}
 \bigl(\mathcal H^j(\psi_F(\Cring_X))_x\otimes_A A_{\mathfrak q}\bigr).
\end{equation}
The cohomology stalks are finitely generated torsion $A$-modules. Since $A_{\mathfrak q}$ is a discrete valuation ring, the displayed lengths are finite. They are constant on a finite adapted algebraic stratification, so $\varphi_{\mathfrak q}\in\mathcal F(X)$; moreover, $\varphi_{\mathfrak q}=0$ for $\mathfrak q\notin\mathfrak P_F$.

The \emph{divisor-valued stalk Euler characteristic} is
\begin{equation}\label{eq:mono:chi}
\chi_{\mathrm{st}}^A(\psi_F(\Cring_X))
=\sum_{\mathfrak q\in\mathfrak P_F}
 \varphi_{\mathfrak q}\otimes[V(\mathfrak q)]
\in\mathcal F_A(X).
\end{equation}
The same construction applies to any constructible $A$-complex with finitely generated torsion cohomology stalks. All divisor sums appearing below have finite support.

\subsection{Coefficient cycles and the local index comparison}
For a height-one prime $\mathfrak p\subset R$, consider $\Psi_{F,\mathfrak p}(\mathcal O_X)$ from~\eqref{eq:localized-Psi}. It is relative holonomic and torsion over $R_{\mathfrak p}$, so its parameter support is contained in the closed point of $\Spec R_{\mathfrak p}$. For an adapted Whitney stratification $X=\bigsqcup_\beta X_\beta$, write
\begin{equation}\label{eq:mono:cycle}
\CCrel\bigl(\Psi_{F,\mathfrak p}(\mathcal O_X)\bigr)
=\sum_\beta m_\beta(\mathfrak p)
 [\overline{T^*_{X_\beta}X}\times\{\mathfrak p\}].
\end{equation}
Define its \emph{coefficient cycle} by
\[
\Lambda_{\mathfrak p}
=\sum_\beta m_\beta(\mathfrak p)[\overline{T^*_{X_\beta}X}]
\in\mathcal L(X).
\]
This cycle is independent of the stratification and filtration. It is effective, and is nonzero whenever $\Psi_{F,\mathfrak p}(\mathcal O_X)\neq0$.

\begin{proposition}\label{prop:mono:translation}
For $\mathbf b\in\mathbb Z^r$, one has $\Lambda_{\tau_{\mathbf b}(\mathfrak p)}=\Lambda_{\mathfrak p}$, where $\tau_{\mathbf b}(h)(\svect)=h(\svect+\mathbf b)$.
\end{proposition}

\begin{proof}
For sufficiently large $k$, stabilization of the maximal and minimal extensions gives
\[
\Psi_{F,\mathfrak p}(\mathcal O_X)
=\bigl(M_F^{-k\one-\mathbf b,k\one-\mathbf b}\bigr)_{\mathfrak p},\qquad
\Psi_{F,\tau_{\mathbf b}(\mathfrak p)}(\mathcal O_X)
=\bigl(M_F^{-k\one,k\one}\bigr)_{\tau_{\mathbf b}(\mathfrak p)}.
\]
Translation of the parameters identifies these localized modules semilinearly and leaves the $T^*X$ factor unchanged. By~\eqref{eq:relative-localization}, the multiplicities in their characteristic cycles agree.
\end{proof}

For a primitive $L\in\mathbb Z_{\geq0}^r\setminus\{0\}$ and $\alpha\in\mathbb Q$, put
\begin{equation}\label{eq:mono:paired}
\mathfrak p(L,\alpha)=(L\cdot\svect+\alpha),\qquad
\mathfrak q(L,\alpha)=(\tvect^L-e^{2\pi\sqrt{-1}\alpha}).
\end{equation}
These are height-one primes. Call $\mathfrak p$ \emph{active} if $\Psi_{F,\mathfrak p}(\mathcal O_X)\neq0$. Theorem~\ref{thm:support-Bernstein-Sato} gives an active representative $\mathfrak p(\mathfrak q)$ of the form~\eqref{eq:mono:paired} for every $\mathfrak q\in\mathfrak P_F$.

Any two such representatives for the same $\mathfrak q$ differ by an integral parameter translation: their constants differ by an integer, and primitivity of $L$ gives $L\cdot\mathbb Z^r=\mathbb Z$. Hence Proposition~\ref{prop:mono:translation} makes
\begin{equation}\label{eq:mono:theta}
\Theta_F=\sum_{\mathfrak q\in\mathfrak P_F}
 \Lambda_{\mathfrak p(\mathfrak q)}\otimes[V(\mathfrak q)]
\in\mathcal L_A(X)
\end{equation}
independent of the representatives.

\begin{lemma}\label{lem:mono:index}
With notations as above,
\[
\CC(\varphi_{\mathfrak q})=(-1)^n\Lambda_{\mathfrak p(\mathfrak q)}
\qquad(\mathfrak q\in\mathfrak P_F).
\]
\end{lemma}
\begin{proof}

For the fixed Whitney stratification with connected strata,
choose a complex normal slice $N_\beta$ to $X_\beta$ at
$x_\beta$, a sufficiently small closed ball
$B_\varepsilon(x_\beta)$, and a generic holomorphic function
$g_\beta:(N_\beta,x_\beta)\to(\mathbb C,0)$.
For $\gamma\neq\beta$ with
$X_\beta\subseteq\overline{X_\gamma}$, define
\[
c_{\beta\gamma}
:=\chi_c\!\left(
B_\varepsilon(x_\beta)\cap N_\beta\cap X_\gamma
\cap g_\beta^{-1}(\delta)
\right),
\qquad 0<|\delta|\ll\varepsilon\ll1.
\]
This is the compactly supported Euler characteristic of the
complex link of $X_\beta$ in $X_\gamma$, and is independent
of the sufficiently small generic choices.
We set $c_{\beta\beta}=0$ and $c_{\beta\gamma}=0$ whenever
$X_\beta\not\subseteq\overline{X_\gamma}$. Let $d_\beta=dim X_\beta$.By \cite[\S1.3]{Wu26}, we have \begin{equation}\label{eq:mono:index-input}
m_\beta(\mathfrak p(\mathfrak q))
 =(-1)^{d_\beta+n}
 \left(\varphi_{\mathfrak q}(x_\beta)
       -\sum_{X_\beta\subseteq \bar{X_\gamma} }
         c_{\beta\gamma}\varphi_{\mathfrak q}(x_\gamma)\right).
\end{equation}
where $x_\beta\in X_\beta$ and $x_\gamma \in X_\gamma$. By definition of local Euler obstruction function, $\CC(\varphi_{\mathfrak q})=(-1)^n\Lambda_{\mathfrak p(\mathfrak q)}$
\end{proof}
\begin{theorem}\label{thm:mono:comparison}
The two divisor-valued invariants satisfy
\begin{equation}\label{eq:mono:comparison}
\CC_A^{-1}(\Theta_F)=(-1)^n\chi_{\mathrm{st}}^A(\psi_F(\Cring_X)).
\end{equation}
\end{theorem}

\begin{proof}
Apply Lemma~\ref{lem:mono:index} coefficientwise:
\[
\CC_A\bigl(\chi_{\mathrm{st}}^A(\psi_F(\Cring_X))\bigr)
=\sum_{\mathfrak q\in\mathfrak P_F}
 \CC(\varphi_{\mathfrak q})\otimes[V(\mathfrak q)]
=(-1)^n\Theta_F.
\]
Then apply $\CC_A^{-1}$.
\end{proof}

\subsection{Detection of monodromy support}
Since $A$ is a unique factorization domain, every divisor on $\mathbb T$ is principal. Define $\zeta_{F,x}^{\mathrm{mon}}\in\operatorname{Frac}(A)^\times/A^\times$ by
\begin{equation}\label{eq:mono:zeta}
\operatorname{div}\bigl(\zeta_{F,x}^{\mathrm{mon}}\bigr)
=\chi_{\mathrm{st}}^A(\psi_F(\Cring_X))(x)
=\sum_{\mathfrak q\in\mathfrak P_F}
 \varphi_{\mathfrak q}(x)[V(\mathfrak q)].
\end{equation}
Equivalently, a representative is
\[
\prod_{\mathfrak q=\mathfrak q(L,\alpha)\in\mathfrak P_F}
\bigl(\tvect^L-e^{2\pi\sqrt{-1}\alpha}\bigr)^{\varphi_{\mathfrak q}(x)}.
\]
Here $PZ(\zeta_{F,x}^{\mathrm{mon}})$ denotes the support of its divisor, that is, the union of its zero and pole divisors. This is independent of the representative.

\begin{corollary}\label{cor:mono:detection}
One has
\[
S(F)=\bigcup_{x\in X}PZ(\zeta_{F,x}^{\mathrm{mon}}).
\]
\end{corollary}

\begin{proof}
Every prime divisor contributing to~\eqref{eq:mono:zeta} is contained in $S(F)$, which gives one inclusion. Conversely, fix $\mathfrak q\in\mathfrak P_F$. Since $\mathfrak p(\mathfrak q)$ is active, $\Lambda_{\mathfrak p(\mathfrak q)}\neq0$. Lemma~\ref{lem:mono:index} and injectivity of $\CC$ imply that $\varphi_{\mathfrak q}$ is not identically zero. Thus $\varphi_{\mathfrak q}(x)\neq0$ for some $x$, and $V(\mathfrak q)$ occurs in the divisor of $\zeta_{F,x}^{\mathrm{mon}}$. Every irreducible component of $S(F)$ is therefore detected.
\end{proof}
\subsection{Proper pushforward}
Let $\mu:Y\to X$ be the log resolution fixed in Section~\ref{lower}, and put $F_Y=F\circ\mu$. 
For a constructible function $\varphi=\sum_\beta a_\beta1_{Y_\beta}$, define
\[
(\mu_*\varphi)(x)
=\sum_\beta a_\beta\chi_c(Y_\beta\cap\mu^{-1}(x)).
\]
Additivity of $\chi_c$ makes this independent of the chosen stratification. Extend $\mu_*$ to $\mathcal F_A(Y)$ by acting on each coefficient.

\begin{theorem}\label{lem:mono:proper}
There is a natural quasi-isomorphism
\begin{equation}\label{eq:mono:proper-base-change}
R\mu_*\psi_{F_Y}(\Cring_Y)\simeq\psi_F(\Cring_X).
\end{equation}
Moreover,
\begin{equation}\label{eq:mono:proper}
\mu_*\chi_{\mathrm{st}}^A(\psi_{F_Y}(\Cring_Y))
=\chi_{\mathrm{st}}^A(\psi_F(\Cring_X))
\qquad\text{in }\mathcal F_A(X).
\end{equation}
\end{theorem}

\begin{proof}
Let $i':\mu^{-1}D\hookrightarrow Y$ and $j':U\hookrightarrow Y$. Proper base change gives
\[
R\mu_*i'_*i'^{-1}\simeq i_*i^{-1}R\mu_*.
\]
Using $\mu\circ j'=j$ and $F_Y|_U=F|_U$, we obtain
\[
\begin{aligned}
R\mu_*\psi_{F_Y}(\Cring_Y)
&\simeq R\mu_*i'_*i'^{-1}Rj'_*(F|_U)^{-1}\mathcal L_A\\
&\simeq i_*i^{-1}R\mu_*Rj'_*(F|_U)^{-1}\mathcal L_A\\
&\simeq i_*i^{-1}Rj_*(F|_U)^{-1}\mathcal L_A
\simeq\psi_F(\Cring_X).
\end{aligned}
\]

For the second assertion, fix a height-one prime $\mathfrak q\subset A$. Choose a finite algebraic stratification $Y=\bigsqcup_\beta Y_\beta$ adapted to the cohomology sheaves of $\psi_{F_Y}(\Cring_Y)$, and let $y_\beta\in Y_\beta$. Put
\[
\varphi^Y_{\mathfrak q}(y)
=\sum_j(-1)^j\ell_{A_{\mathfrak q}}
 \bigl(\mathcal H^j(\psi_{F_Y}(\Cring_Y))_y\otimes_A A_{\mathfrak q}\bigr).
\]
Let $T=\mu^{-1}(x)$. Since $\mu$ is proper, the stalk formula and~\eqref{eq:mono:proper-base-change} give
\[
(\psi_F(\Cring_X))_x\simeq R\Gamma_c(T,\psi_{F_Y}(\Cring_Y)|_T).
\]
Triangulate $T$ compatibly with the strata and the constructible cohomology sheaves; see \cite{Hironaka1975}. On each open simplex these sheaves are constant. Exactness of localization and additivity of length over the discrete valuation ring $A_{\mathfrak q}$ therefore give
\[
\begin{aligned}
\varphi_{\mathfrak q}(x)
&=\sum_{p,j}(-1)^{p+j}\ell_{A_{\mathfrak q}}
 \bigl(H_c^p(T,\mathcal H^j(\psi_{F_Y}(\Cring_Y))|_T)
                          \otimes_A A_{\mathfrak q}\bigr)\\
&=\sum_\beta\chi_c(T\cap Y_\beta)
 \sum_j(-1)^j\ell_{A_{\mathfrak q}}
 \bigl(\mathcal H^j(\psi_{F_Y}(\Cring_Y))_{y_\beta}
                          \otimes_A A_{\mathfrak q}\bigr)\\
&=\sum_\beta\chi_c(T\cap Y_\beta)\varphi^Y_{\mathfrak q}(y_\beta)
=(\mu_*\varphi^Y_{\mathfrak q})(x).
\end{aligned}
\]
Only finitely many degrees and strata contribute. Summing over the height-one primes with nonzero coefficients proves~\eqref{eq:mono:proper}.
\end{proof}

\subsection{The normal-crossings calculation}\label{sec:mono-nc}
For $E\in\mathcal E$, set
\[
N_E=(N_{E,1},\ldots,N_{E,r}),\qquad
P_E(\tvect)=\tvect^{N_E}-1,\qquad
E^\circ=E\setminus\bigcup_{E'\neq E}E'.
\]
For $I\subseteq\mathcal E$, write $E_I=\bigcap_{E\in I}E$. Empty intersections contribute zero to the sums below. If $E_I$ is disconnected, $[T^*_{E_I}Y]$ denotes the sum over its connected components.

\begin{theorem}[Normal-crossings computation]\label{thm:mono:normal-crossings}
Let $\Theta_{F_Y}$ and $\Lambda^Y_{\mathfrak p}$ be the coefficient cycles on $Y$ defined as in~\eqref{eq:mono:cycle} and~\eqref{eq:mono:theta}. Then
\begin{equation}\label{eq:mono:nc-theta}
\Theta_{F_Y}
=\sum_{\varnothing\neq I\subseteq\mathcal E}
 [T^*_{E_I}Y]\otimes\left(\sum_{E\in I}\operatorname{div}(P_E)\right).
\end{equation}
More precisely, for $\mathfrak q\in\mathfrak P_{F_Y}$,
\begin{equation}\label{eq:mono:nc-lambda}
\Lambda^Y_{\mathfrak p(\mathfrak q)}
=\sum_{\varnothing\neq I\subseteq\mathcal E}
 \left(\sum_{E\in I}\operatorname{ord}_{\mathfrak q}(P_E)\right)[T^*_{E_I}Y],
\end{equation}
where $\operatorname{ord}_{\mathfrak q}(P_E)\in\{0,1\}$. Consequently,
\begin{equation}\label{eq:mono:nc-chi}
\chi_{\mathrm{st}}^A(\psi_{F_Y}(\Cring_Y))
=-\sum_{E\in\mathcal E}
 1_{E^\circ}\otimes\operatorname{div}(P_E).
\end{equation}
\end{theorem}

\begin{proof}
We compute the contribution of each prime divisor and then apply the local index comparison.

\smallskip\noindent\emph{The parameter divisors.}
Write $N_E=dL$ with $d\geq1$ and $L$ primitive. Then
\[
P_E(\tvect)=\prod_{\xi^d=1}(\tvect^L-\xi).
\]
Thus every prime factor of $P_E$ occurs with multiplicity one.
The calculation of $B^{\mathbf 1}(F_Y)$ and Theorem~\ref{thm:support-Bernstein-Sato}, applied locally on $Y$, give
\[
S(F_Y)=\bigcup_{E\in\mathcal E}V(P_E).
\]
For a paired representative $\mathfrak p=\mathfrak p(\mathfrak q)$ as in~\eqref{eq:mono:paired},
\begin{equation}\label{eq:mono:nc-resonance}
\operatorname{ord}_{\mathfrak q}(P_E)=1
\quad\Longleftrightarrow\quad
L_E(\svect)+h\in\mathfrak p\text{ for a unique }h\in\mathbb Z.
\end{equation}
Indeed, if $\mathfrak p=(L\cdot\svect+\alpha)$ and $N_E=dL$, both conditions say that $d\alpha\in\mathbb Z$, and $h=d\alpha$. If the normal vectors are not proportional, neither condition holds.

\smallskip\noindent\emph{A single coordinate shift.}
Work on a coordinate neighborhood $V\subseteq Y$ where the divisor components are $E_a=(y_a=0)$, $1\leq a\leq k$, and
\[
f_i\circ\mu=u_i\prod_{a=1}^k y_a^{N_{E_a,i}}.
\]

By the same argument as in the proof of
Proposition~\ref{prop:local-normal-crossing}, we may omit
the unit factors $u_i$ when computing the relative
characteristic cycles.

For $\mathbf b=(b_1,\ldots,b_k)\in\mathbb Z^k$, set
\[
\Mmod^{\mathbf b}
=\Dmod_V[\svect]_{\mathfrak p}
 \left(\prod_{a=1}^k y_a^{b_a}F_Y^{\svect}\right).
\]
Writing $\mathscr D=\Dmod_V[\svect]_{\mathfrak p}$, the coordinatewise cyclic presentation is
\[
\Mmod^{\mathbf b}\simeq
\frac{\mathscr D}
 {\displaystyle\sum_{j=1}^k\mathscr D(y_j\partial_{y_j}-L_{E_j}(\svect)-b_j)
  +\sum_{c=k+1}^n\mathscr D\partial_{y_c}}.
\]
The quotient $\Mmod^{\mathbf b}/\Mmod^{\mathbf b+\mathbf e_a}$ is obtained by adding the relation $y_a=0$. Since $\partial_{y_a}y_a=y_a\partial_{y_a}+1$, it is annihilated by $L_{E_a}(\svect)+b_a+1$. It vanishes if this linear form is not in $\mathfrak p$. Otherwise the linear form generates $\mathfrak pR_{\mathfrak p}$, and the quotient is a module over
\[
K=R_{\mathfrak p}/\mathfrak pR_{\mathfrak p}
=\operatorname{Frac}(R/\mathfrak p).
\]
The tensor product of the one-coordinate order filtrations gives
\[
\gr^{\mathrm{rel}}
 \bigl(\Mmod^{\mathbf b}/\Mmod^{\mathbf b+\mathbf e_a}\bigr)
\simeq
\frac{(\mathcal O_V\otimes_{\Cring}K)[\xi_1,\ldots,\xi_n]}
 {(y_a,\ y_j\xi_j\ (1\leq j\leq k,\ j\neq a),\ \xi_c\ (c>k))}.
\]
The ideal in the denominator is the intersection
\[
\bigcap_{\substack{J\subseteq\{1,\ldots,k\}\\a\in J}}
 (y_j\ (j\in J),\ \xi_\ell\ (\ell\notin J)).
\]
Each component is a conormal bundle and is generically reduced. Hence, when the quotient is nonzero, its characteristic cycle is
\begin{equation}\label{eq:mono:coordinate-cycle}
\CCrel\bigl(\Mmod^{\mathbf b}/\Mmod^{\mathbf b+\mathbf e_a}\bigr)
=\sum_{\substack{J\subseteq\{1,\ldots,k\}\\a\in J}}
 [T^*_{\bigcap_{j\in J}(E_j\cap V)}V\times\{\mathfrak p\}].
\end{equation}

\smallskip\noindent\emph{The stabilized quotient.}
For $m\gg0$, Theorems~\ref{thm:maximal-extension} and~\ref{thm:minimal-extension} identify
\[
\Psi_{F_Y,\mathfrak p}(\mathcal O_Y)|_V
\simeq\Mmod^{\mathbf b^-}/\Mmod^{\mathbf b^+},
\qquad b_a^\pm=\pm mL_{E_a}(\one).
\]
Choose a monotone path from $\mathbf b^-$ to $\mathbf b^+$, increasing one coordinate by one at each step. It gives a filtration whose successive quotients are those in~\eqref{eq:mono:coordinate-cycle}. For fixed $a$, condition~\eqref{eq:mono:nc-resonance} shows that exactly one step contributes if $\operatorname{ord}_{\mathfrak q}(P_{E_a})=1$, and no step contributes otherwise. Additivity of the characteristic cycles gives~\eqref{eq:mono:nc-lambda} locally. These local calculations agree on overlaps, so the identity holds on $Y$. Summing over $\mathfrak q$ proves~\eqref{eq:mono:nc-theta}.

\smallskip\noindent\emph{The constructible function.}
Every nonempty $E_I$ is smooth of dimension $n-|I|$, so~\eqref{eq:mono:cc-map} and inclusion--exclusion give
\[
\begin{aligned}
\CC_A^{-1}(\Theta_{F_Y})
&=\sum_{\varnothing\neq I\subseteq\mathcal E}
 (-1)^{n-|I|}1_{E_I}\otimes
 \left(\sum_{E\in I}\operatorname{div}(P_E)\right)\\
&=(-1)^{n-1}\sum_{E\in\mathcal E}
 \left(\sum_{I\ni E}(-1)^{|I|-1}1_{E_I}\right)
 \otimes\operatorname{div}(P_E)\\
&=(-1)^{n-1}\sum_{E\in\mathcal E}
 1_{E^\circ}\otimes\operatorname{div}(P_E).
\end{aligned}
\]
Theorem~\ref{thm:mono:comparison}, applied locally on $Y$, proves~\eqref{eq:mono:nc-chi}.
\end{proof}

\subsection{The multivariable A'Campo formula}

\begin{theorem}[Multivariable A'Campo formula]\label{thm:mono:acampo}
For every $x\in X$,
\begin{equation}\label{eq:mono:acampo-divisor}
\operatorname{div}\bigl(\zeta_{F,x}^{\mathrm{mon}}\bigr)
=-\sum_{E\in\mathcal E}
 \chi_c(E^\circ\cap\mu^{-1}(x))\operatorname{div}(P_E).
\end{equation}
Equivalently,
\begin{equation}\label{eq:mono:acampo-product}
\zeta_{F,x}^{\mathrm{mon}}(\tvect)
\doteq\prod_{E\in\mathcal E}
 \bigl(\tvect^{N_E}-1\bigr)^{-\chi_c(E^\circ\cap\mu^{-1}(x))},
\end{equation}
where $\doteq$ denotes equality up to a unit of $A$. The resulting divisor is independent of the log resolution.
\end{theorem}

\begin{proof}
Theorems~\ref{lem:mono:proper} and~\ref{thm:mono:normal-crossings} give
\[
\chi_{\mathrm{st}}^A(\psi_F(\Cring_X))
=-\sum_{E\in\mathcal E}
 (\mu_*1_{E^\circ})\otimes\operatorname{div}(P_E).
\]
By definition,
\[
(\mu_*1_{E^\circ})(x)=\chi_c(E^\circ\cap\mu^{-1}(x)).
\]
Evaluating at $x$ and using~\eqref{eq:mono:zeta} proves~\eqref{eq:mono:acampo-divisor}; factorization gives~\eqref{eq:mono:acampo-product}. Independence follows from the intrinsic definition of the left-hand side.
\end{proof}

\bibliographystyle{amsplain}
\bibliography{references}
\end{document}